\documentclass[a4paper,reqno,12pt]{amsart}
\usepackage[english]{babel}
\usepackage[T1]{fontenc}
\usepackage[utf8]{inputenc}
\usepackage{mathptmx,amsmath,amssymb}
\usepackage{fullpage,microtype,needspace}
\usepackage[hidelinks]{hyperref}
\hypersetup{pdftitle={An elementary proof that sets of cardinality seven are not sum-dominant},pdfauthor={Thai Duong Do and Van Thien Nguyen}}
\newtheorem{theorem}{Theorem}[section]
\newtheorem{lemma}[theorem]{Lemma}
\newtheorem{proposition}[theorem]{Proposition}

\numberwithin{equation}{section}
\newcommand{\R}{\mathbb R}

\allowdisplaybreaks[1]

\begin{document}
\author{Thai Duong Do\textit{$^{1,2}$}}
\address{$^{1}$Institute for Artificial Intelligence, VNU University of Engineering and Technology, Hanoi, Vietnam.}
\address{$^{2}$Department of Mathematics, National University of Singapore, 10 Lower Kent Ridge Road, Singapore, 119076, Singapore.}
\email{dtduong@vnu.edu.vn, dtduong@nus.edu.sg}
\author{Van Thien Nguyen\textit{$^{3}$}}
\address{$^{3}$FPT University, Education and Training Zone - Hoa Lac High-Tech Park - Km29 Thang Long Boulevard, Hoa Lac Commune, Hanoi City, Viet Nam.}
\email{thiennv15@fe.edu.vn}
\title[Sets of cardinality seven are not sum-dominant]{An elementary proof that sets of cardinality seven are not sum-dominant}
\date{}

\begin{abstract}
We give a self-contained elementary proof of the known result that
every set $A\subset\R$ of cardinality seven satisfies
$|A+A|\le |A-A|$. The argument adapts Hegarty's method, using
representation counts and the largest positive differences. An exact
counting identity, two applications of the Cauchy--Schwarz inequality,
and an analysis of a symmetric six-element set with one point added
complete the proof, with no computer enumeration. This answers in the
affirmative a question of Chu for the seven-element case.
\end{abstract}
\maketitle

\section{Introduction}

For a finite set $A\subset\R$, write
\[
A+A=\{a+b:a,b\in A\},\qquad A-A=\{a-b:a,b\in A\}.
\]
The set $A$ is called \emph{sum-dominant}, or an \emph{MSTD set}, if
$|A+A|>|A-A|$. Although a difference set is symmetric about zero,
this does not imply that its cardinality is at least that of the sumset. The first explicit example of an MSTD set of cardinality eight is
due to Nathanson~\cite{Nathanson07}, who constructed the set
$\{0,2,3,4,7,11,12,14\}$ with $|A+A|=26$ and $|A-A|=25$.
The existence and the construction of MSTD sets have been studied in
additive number theory; see, for example, Hegarty~\cite{Hegarty07}.
Recent work includes constructions of further families of MSTD sets
by Kumar, Mohan, and Pandey~\cite{KMP24}, who also disproved a
conjecture about a family proposed by Chu and collaborators.
Herrmann, Hill, Phillips, Flores, Miller, and Senger~\cite{HH26}
constructed increasing sequences of sets that alternate between
sum-dominance and difference-dominance while preserving gaps within
the range of each earlier set.

The same definition applies to finite subsets of any abelian group.
Zhao~\cite{Zhao10} studied the asymptotic enumeration of
MSTD sets in finite abelian groups. Penman and Wells~\cite{PW14}
classified the finite abelian groups that contain sum-dominant sets,
and those that contain restricted-sum-dominant sets. For the latter,
the sumset is formed using distinct summands and is required to have
larger cardinality than the ordinary difference set.

Hegarty~\cite[Theorem~1]{Hegarty07} proved by a computer-assisted
argument that no seven-element set of integers is sum-dominant, and
classified the MSTD sets of cardinality eight up to affine equivalence.
Nathanson~\cite{Nathanson18} showed that every finite real set has
a Freiman-isomorphic integer model that preserves the cardinalities
of both the sumset and the difference set; in particular, he deduced
the nonexistence of real MSTD sets of cardinality at most seven.
Chu~\cite{Chu20} gave a computer-free proof for cardinality six and,
in Section~7 of that paper, asked for such a proof in the seven-element
case.

In this paper, we prove the following result by an elementary argument.

\begin{theorem}\label{thm:main}
Let $A\subset\R$ be a set of cardinality seven. Then
\[
|A+A|\le |A-A|.
\]
\end{theorem}

The conclusion of Theorem~\ref{thm:main}, including its formulation
for real sets, is therefore known. Our contribution is a self-contained
elementary proof of the seven-element case, of the type requested in
\cite[Section~7]{Chu20}. All steps are proved directly over $\R$.

The method is inspired by the combinatorial argument in
\cite[Section~2]{Hegarty07}. Hegarty considers sums with three
representations on six distinct elements, counts pairs of equal
positive differences, and separates cases according to whether the
largest difference is the only one represented uniquely; his analysis
of the next largest differences yields the configurations used in his
classification. We adapt these steps to seven elements, retaining
the contribution of repeated summands in an exact counting identity
and using two Cauchy--Schwarz estimates to handle the case in which
no symmetric six-element subset exists.

Let us outline the argument. A set $B\subset\R$ is symmetric if
$B=c-B$ for some $c\in\R$. Adjoining one point to a symmetric set is a standard construction of MSTD sets; see Nathanson~\cite[Section~2]{Nathanson07}. We first show directly that this construction cannot produce an MSTD set when $|B|=6$. Suppose then that a seven-element MSTD set contains no such symmetric subset.
This bounds the number of representations of each sum by pairs of distinct elements, and an exact counting identity together with the Cauchy--Schwarz inequality implies that only the largest positive difference can have a unique representation. The next two largest
differences then determine enough of the set to produce a contradiction by a second application of Cauchy--Schwarz.

Section~\ref{sec:counts} records the counting identity.
Section~\ref{sec:symmetry} treats symmetric sets with one point added.
We prove Theorem~\ref{thm:main} in Section~\ref{sec:proof}.

\Needspace{8\baselineskip}
\section{Counting representations}\label{sec:counts}

Representations of a positive difference are ordered pairs, whereas representations of a sum are unordered unless otherwise stated; repeated summands are allowed throughout.

Let $A\subset\R$ be finite and nonempty, and put $n=|A|$. Define
\[
D=(A-A)\cap(0,\infty),\qquad k=|D|,
\qquad m_d=\#\{(a,b)\in A^2:a-b=d\}\quad(d\in D).
\]
Every unordered pair of distinct elements gives exactly one positive
difference. Consequently,
\begin{equation}\label{eq:basic-differences}
\sum_{d\in D}m_d=\binom n2,
\qquad |A-A|=2k+1.
\end{equation}
For $s\in A+A$, let
\[
p_s=\#\{(a,b)\in A^2:a<b,\ a+b=s\},\qquad
\varepsilon_s=
\begin{cases}
1,&s/2\in A,\\
0,&s/2\notin A.
\end{cases}
\]
Thus $p_s$ counts representations with distinct summands, and
$p_s+\varepsilon_s$ is the total number of unordered representations.
Set
\begin{equation}\label{eq:RE}
R=\frac{n(n+1)}2-|A+A|,
\qquad E=\sum_{d\in D}\binom{m_d}{2}.
\end{equation}
Here $R$ measures the loss of distinct sums from the maximum possible
number, while $E$ counts unordered pairs of distinct representations
of the same positive difference. In general, $E$ is not equal to
$\binom n2-k=\sum_{d\in D}(m_d-1)$: a difference of multiplicity
three contributes three to $E$ but only two to the latter sum.

The following identity is a form of the usual equality between the
sum and difference expressions for additive energy. We give a direct
counting proof and keep track of representations with equal summands.

\begin{lemma}\label{lem:counting}
With the notation above,
\begin{align}
R&=\sum_{s\in A+A}(p_s+\varepsilon_s-1),\label{eq:Rsum}\\
E&=\sum_{s\in A+A}\left(2\binom{p_s}{2}
                  +p_s\varepsilon_s\right).\label{eq:collision}
\end{align}
If $p_s\le2$ for all $s\in A+A$, then
\begin{equation}\label{eq:exact}
E=2R-h,
\qquad h=\#\{s\in A+A:p_s=1,\ \varepsilon_s=1\}.
\end{equation}
\end{lemma}

\begin{proof}
There are $n(n+1)/2$ unordered pairs of elements of $A$, including
pairs with equal entries. Partitioning them according to their sum
gives~\eqref{eq:Rsum}.

To prove~\eqref{eq:collision}, consider two distinct representations
of a positive difference,
\[
a-b=c-e>0.
\]
They give the sum equality $a+e=b+c$. The two pairs cannot have the
same larger entry or the same smaller entry, because either equality
would make the representations identical. Hence their union has
either three or four elements.

If there are four elements, write them as $w<x<y<z$. The only
equality between two disjoint pairs with the same sum is
$w+z=x+y$: the smallest element must be paired with the largest.
Conversely, this equality gives exactly two unordered pairs of
representations of positive differences,
\[
x-w=z-y,\qquad y-w=z-x.
\]
For a fixed sum $s$, there are $\binom{p_s}{2}$ choices of two
distinct representations with unequal summands. Their summand pairs
are disjoint, since a common summand would force the other summands
to agree. Thus their contribution is $2\binom{p_s}{2}$.

If there are three elements, write them as $w<y<z$. The equality is
$w+z=2y$, and it gives exactly one pair of representations,
$y-w=z-y$. For the sum $s=2y$, there are $p_s\varepsilon_s$ such
choices. These two cases exhaust the pairs counted by $E$, proving
\eqref{eq:collision}.

Suppose now that $p_s\le2$. Since $s\in A+A$, the pair
$(p_s,\varepsilon_s)$ cannot be $(0,0)$. Its five possible values
give the following contributions:
\[
\begin{array}{c|ccccc}
(p_s,\varepsilon_s)&(0,1)&(1,0)&(1,1)&(2,0)&(2,1)\\ \hline
p_s+\varepsilon_s-1&0&0&1&1&2\\
2\binom{p_s}{2}+p_s\varepsilon_s&0&0&1&2&4
\end{array}
\]
Twice the contribution to $R$ minus the contribution to $E$ is one
exactly for $(1,1)$, and is zero otherwise. Summing proves~\eqref{eq:exact}.
\end{proof}

We shall also use the following elementary consequence of
Cauchy--Schwarz. If $r\ge1$ and $b_1,\ldots,b_r$ are positive
integers with $b_1+\cdots+b_r=M$, then
\begin{equation}\label{eq:cauchy}
\sum_{i=1}^r\binom{b_i}{2}
=\frac12\left(\sum_{i=1}^r b_i^2-M\right)
\ge\frac12\left(\frac{M^2}{r}-M\right).
\end{equation}

\section{A symmetric set with one point added}\label{sec:symmetry}

We first record a restriction on the signs in a finite set whose
elements are sums of elements of the same set.

\begin{lemma}\label{lem:signs}
Let $C\subset\R$ be finite and nonempty. If
\begin{equation}\label{eq:sumfull}
C\subseteq C+C,\qquad 0\notin C+C,
\end{equation}
then $C$ has at least three positive elements and at least three
negative elements.
\end{lemma}

\begin{proof}
The second condition implies $0\notin C$, since otherwise $0=0+0$
would belong to $C+C$.
If $C$ has no positive element, let $v=\max C<0$. Every sum of two
elements of $C$ is at most $2v<v$, contradicting $v\in C+C$.
If $p$ is the only positive element, a sum of two negative elements
is negative, a sum of $p$ and a negative element is smaller than $p$,
and $p+p>p$. Thus $p\notin C+C$, again a contradiction.

Suppose that $C$ has exactly two positive elements $p<q$.
A representation of $q$ cannot have a negative summand, since the
other summand is at most $q$. Both summands must therefore be
positive. Neither can be $q$, so $q=2p$.
Now a representation of $p$ cannot have two positive summands or
two negative summands. Its positive summand cannot be $p$, since
zero is absent. It must be $q$, and the other summand is
$p-q=-p$. Hence $-p\in C$ and $0=p+(-p)\in C+C$, a contradiction.
Applying the same argument to $-C$ gives the assertion about negative
elements, since $-C\subseteq(-C)+(-C)$ and $0\notin(-C)+(-C)$.
\end{proof}

\begin{proposition}\label{prop:symmetric}
Let $B\subset\R$ be symmetric with $|B|=6$, and let $x\notin B$.
Then $B\cup\{x\}$ is not sum-dominant.
\end{proposition}

\begin{proof}
Write $B=c-B$. Subtracting $c/2$ from all elements of $B\cup\{x\}$
does not change the cardinalities of its sumset and difference set.
We may therefore assume $B=-B$.
If $x=0$, then $B\cup\{x\}$ is symmetric and has equally many
sums and differences. Otherwise, reflecting all elements if
necessary, we may assume $x>0$.
Put
\[
A=B\cup\{x\},\qquad T=B+B=B-B,\qquad U=(x+B)\setminus T.
\]
Since $B=-B$, we have $T=-T$ and
\begin{equation}\label{eq:adjoin}
\begin{split}
A+A&=T\cup(x+B)\cup\{2x\},\\
A-A&=T\cup(x+B)\cup(-x+B).
\end{split}
\end{equation}
Here the difference $x-x=0$ is already in $T$, since $B$ is nonempty.
If $2x\in T$, these expressions show $A+A\subseteq A-A$.
It remains to consider $2x\notin T$.

In this case $(x+B)\cap(-x+B)=\varnothing$. Indeed, an equality
$x+b=-x+b'$ would imply $2x=b'-b\in B-B=T$.
Moreover, $2x\notin x+B$ because $x\notin B$.
The part of $-x+B$ outside $T$ is $-U$, which is disjoint from $U$.
Thus~\eqref{eq:adjoin} gives
\[
|A+A|=|T|+|U|+1,\qquad |A-A|=|T|+2|U|.
\]
Suppose that $A$ is sum-dominant. Then $1>|U|$, so $U=\varnothing$.
In particular,
\[
x+B\subseteq B+B,\qquad 2x\notin B+B.
\]
For $C=B-x$, we have
\[
(x+B)-2x=C,\qquad (B+B)-2x=C+C.
\]
Thus these two conditions give exactly~\eqref{eq:sumfull}.

By Lemma~\ref{lem:signs}, $C$ has exactly three positive elements
and three negative elements. Write its positive elements as
$0<p<q<r$, and set $t=2x>0$. The map $c\mapsto-t-c$ preserves
$C$, because $B=-B$. It maps the three positive elements to three
distinct negative elements, which must be all its negative elements.
Therefore
\begin{equation}\label{eq:shapeC}
C=\{p,q,r,-p-t,-q-t,-r-t\}.
\end{equation}

We next use the representations of the largest and smallest elements
of $C$. A representation of $r$ cannot contain a negative element:
the other summand is at most $r$. It also cannot contain $r$,
because the other summand would be zero. Hence
$r\in\{2p,p+q,2q\}$.
Similarly, a representation of $-r-t$ must use two negative
elements: adding a positive element to any element of $C$ gives
a number strictly greater than $-r-t$.
Neither summand can equal $-r-t$, because the other would be zero.
Its two summands therefore lie in $\{-p-t,-q-t\}$, and
\[
-r-t\in\{-2p-2t,-p-q-2t,-2q-2t\}.
\]
Consequently,
\begin{equation}\label{eq:threevalues}
r\in\{2p,p+q,2q\},\qquad r-t\in\{2p,p+q,2q\}.
\end{equation}

The three numbers on the right form an increasing arithmetic
progression with common difference $q-p$. Since $t>0$, the two
selected numbers $r-t$ and $r$ are distinct. If they are consecutive
terms, then $t=q-p$. In that case $-p-t=-q\in C$, contrary to
$0\notin C+C$. They must therefore be the first and last terms:
\begin{equation}\label{eq:endpoints}
r=2q,\qquad r-t=2p,\qquad t=2(q-p).
\end{equation}

Finally, consider a representation of $q$. Two negative summands
are impossible. If both summands are positive, each is smaller
than $q$, so both equal $p$ and $q=2p$.
If one summand is negative, the positive summand is greater than
$q$ and must be $r$. The negative summand would then be
$q-r=-q$, which is forbidden by $0\notin C+C$.
Thus $q=2p$. Equation~\eqref{eq:endpoints} now gives
$r=4p$ and $t=2p$. In particular, $-q-t=-r\in C$, the final
contradiction to $0\notin C+C$.
\end{proof}

\Needspace{12\baselineskip}
\section{Proof of the main theorem}\label{sec:proof}

We now treat seven-element sets that contain no symmetric subset
of cardinality six. The argument is by contradiction. Assuming that
no symmetric six-element subset exists, we first observe that every sum has at most two representations with distinct summands; otherwise three disjoint representations would produce such a subset. The counting identity of Lemma~\ref{lem:counting} then applies, and the
Cauchy--Schwarz estimate forces all but the largest difference to have multiplicity at least two. This pins down the second largest
difference as $L-u$ and constrains the third, leaving one of two reflected configurations. In the configuration $a_2=2u$,
$a_4<L-2u$, the sum $2u$ has a repeated summand, which sharpens the bound on $E$; a second Cauchy--Schwarz estimate then yields an impossible quadratic inequality.

The remainder of this section makes this outline precise.

\begin{proposition}\label{prop:core}
Let $A\subset\R$ have cardinality seven. If $A$ contains no symmetric subset of cardinality six, then $|A+A|\le |A-A|$.
\end{proposition}

\begin{proof}
Suppose, to the contrary, that $A$ is sum-dominant.
Use the notation of Section~\ref{sec:counts}. If a sum $s$ had three representations with distinct summands, the corresponding pairs would be disjoint. Their six elements would form a subset preserved by $a\mapsto s-a$, contrary to the assumption. Hence $p_s\le2$ for every $s\in A+A$, and Lemma~\ref{lem:counting} gives
\begin{equation}\label{eq:seven-counts}
E=2R-h,\qquad \sum_{d\in D}m_d=21,\qquad R=28-|A+A|.
\end{equation}
There are at least six positive differences, obtained by subtracting the smallest element of $A$ from the other six elements. Thus $k\ge6$, so both $k-2$ and $k-3$ used below are positive.
Since $A$ is sum-dominant and cardinalities are integers,
\[
28-R=|A+A|\ge |A-A|+1=2k+2.
\]
It follows that
\begin{equation}\label{eq:Rbound}
R\le26-2k,\qquad E\le2R\le52-4k.
\end{equation}

We first prove that exactly one positive difference has a unique representation. Suppose that two distinct differences have
multiplicity one. Remove their multiplicities from the sum in \eqref{eq:seven-counts}. The remaining $k-2$ multiplicities have sum $19$. Since $\binom12=0$, inequality~\eqref{eq:cauchy} yields
\[
E\ge\frac12\left(\frac{19^2}{k-2}-19\right).
\]
Combining this with~\eqref{eq:Rbound}, and using $k-2>0$, gives
\[
361\le(123-8k)(k-2),
\qquad 8k^2-139k+607\le0.
\]
However,
\begin{equation}\label{eq:first-square}
8k^2-139k+607
=8\left(k-\frac{139}{16}\right)^2+\frac{103}{32}>0.
\end{equation}
This is a contradiction. The difference between the largest and smallest elements always has a unique representation, so it is the only such positive difference.

Translate and arrange the elements as
\[
A=\{0=a_0<a_1<a_2<a_3<a_4<a_5<a_6=L\}.
\]
Translation preserves sum and difference cardinalities, difference
multiplicities, and the absence of a symmetric six-element subset.
We determine the two largest positive differences below $L$.
After the pair $(L,0)$ is removed, the largest remaining difference is
\[
d_2=\max\{L-a_1,a_5\}.
\]
Only the pairs $(L,a_1)$ and $(a_5,0)$ can attain this maximum.
Indeed, a pair with larger entry $L$ and smaller entry at least
$a_2$ gives a difference strictly below $L-a_1$. A pair whose
larger entry is at most $a_5$ gives a difference at most $a_5$, with equality only for $(a_5,0)$.
Since $d_2$ cannot have a unique representation, these two candidate values must be equal. Put $u=a_1>0$. Then
\begin{equation}\label{eq:second}
a_5=L-u,\qquad d_2=L-u,\qquad m_{L-u}=2.
\end{equation}

After removing the representations of $L$ and $L-u$, the largest remaining difference is
\begin{equation}\label{eq:third}
d_3=\max\{L-a_2,L-2u,a_4\}.
\end{equation}
The candidate pairs are, respectively,
\[
(L,a_2),\qquad (L-u,u),\qquad (a_4,0).
\]
To justify that the list is exhaustive, consider any remaining pair.
If its larger entry is $L$, its smaller entry is at least $a_2$,
with equality only for the first candidate pair.
If its larger entry is $L-u=a_5$, its smaller entry is at least $u$,
with equality only for the second pair. All other pairs have larger
entry at most $a_4$ and smaller entry at least zero, so their difference is at most $a_4$, with equality only for the third pair.
Every noncandidate pair has difference strictly below $d_3$.

The three numbers in~\eqref{eq:third} are positive and smaller
than $L-u$. As $d_3$ is not uniquely represented, the maximum in
\eqref{eq:third} occurs at least twice. If $L-a_2=a_4$, then
\[
0+L=u+(L-u)=a_2+a_4=L.
\]
These pairs use the six distinct elements with indices $0,1,2,4,5,6$.
Their union is preserved by $a\mapsto L-a$, so it is a symmetric
six-element subset. We have excluded this possibility, including the case
where all three values in~\eqref{eq:third} agree.
Thus the remaining possibilities are
\begin{equation}\label{eq:two-configurations}
L-a_2=L-2u>a_4
\quad\hbox{or}\quad
a_4=L-2u>L-a_2.
\end{equation}
Reflection $A\mapsto L-A$ interchanges them. More explicitly, the
new ordered elements are $a'_i=L-a_{6-i}$, so $a'_1=u$,
$a'_2=L-a_4$, and $a'_4=L-a_2$.
In the second case, these satisfy $a'_2=2u$ and
$a'_4<L-2u$. We may therefore assume
\begin{equation}\label{eq:configuration}
a_2=2u,\qquad a_4<L-2u.
\end{equation}
Reflection also preserves the absence of a symmetric six-element
subset. It preserves positive difference multiplicities through
$(a,b)\mapsto(L-b,L-a)$, and it preserves $h$ because its action
on sums is $s\mapsto2L-s$, with distinct and equal summands
remaining distinct and equal, respectively.

It follows from the exhaustive list of candidate pairs that the
three largest positive differences have precisely the following
representations:
\[
\begin{array}{c|c|c}
\text{Difference}&\text{Representing ordered pairs}&\text{Multiplicity}\\ \hline
L&(L,0)&1\\
L-u&(L,u),\ (L-u,0)&2\\
L-2u&(L,2u),\ (L-u,u)&2
\end{array}
\]
For the last row, all pairs with larger entry at most $a_4$ give
a smaller difference. A pair with larger entry $L$ and smaller
entry at least $a_3$ gives a smaller difference because $a_3>2u$.
A pair with larger entry $L-u$ and smaller entry at least $2u$
also gives a smaller difference.

There is one more consequence of~\eqref{eq:configuration}.
The three smallest elements of $A$ are $0,u,2u$, and all remaining
elements exceed $2u$. Since all elements are nonnegative, the only
unordered representations of the sum $2u$ are
\[
2u=0+2u=u+u.
\]
Thus $p_{2u}=1$ and $\varepsilon_{2u}=1$, so $h\ge1$.
Together with~\eqref{eq:seven-counts} and~\eqref{eq:Rbound}, this gives
\begin{equation}\label{eq:improved}
E\le2R-1\le51-4k.
\end{equation}

We now remove the three largest positive differences, of multiplicities $1,2,2$, from the count. Their contribution to $E$ is $\binom{1}{2}+\binom{2}{2}+\binom{2}{2}=2$. The remaining $k-3$ multiplicities have sum $21-1-2-2=16$.
Applying~\eqref{eq:cauchy} to them gives
\[
E\ge2+\frac12\left(\frac{16^2}{k-3}-16\right)
=\frac{128}{k-3}-6.
\]
Combining this with~\eqref{eq:improved}, and using $k-3>0$, we obtain
\[
128\le(57-4k)(k-3),
\qquad 4k^2-69k+299\le0.
\]
This contradicts
\begin{equation}\label{eq:second-square}
4k^2-69k+299
=4\left(k-\frac{69}{8}\right)^2+\frac{23}{16}>0.
\end{equation}
The proposition follows.
\end{proof}

\begin{proof}[Proof of Theorem~\ref{thm:main}]
Let $A\subset\R$ have cardinality seven. If $A$ contains a symmetric
subset $B$ with $|B|=6$, write $A=B\cup\{x\}$ and apply
Proposition~\ref{prop:symmetric}. If it contains no such subset,
apply Proposition~\ref{prop:core}. In either case,
$|A+A|\le |A-A|$.
\end{proof}

\medskip
\noindent\textbf{Acknowledgements.}
This work was completed while the first author was visiting the Department of Mathematics at the National University of Singapore as a Visiting Senior Research Fellow under the Singapore Academies Southeast Asia Fellowship (SASEAF) Programme, within the project ``Math, Sobolev Space'' (Grant No. E-146-00-0039-01). The first author would like to express his sincere gratitude to the Department of Mathematics at NUS for its warm hospitality and excellent working environment. He also gratefully acknowledges the Singapore National Academy of Science (SNAS) and the National Research Foundation (NRF), Singapore, for their financial support through the SASEAF Programme.

\Needspace{24\baselineskip}

\end{document}